\documentclass[11pt]{amsart}

\usepackage[T1]{fontenc}
\usepackage[utf8]{inputenc}
\usepackage{amsmath,amsthm,amsfonts,amssymb}
\usepackage{enumerate}
\usepackage{hyperref}

\hypersetup{colorlinks=true,citecolor=blue,linkcolor=blue,urlcolor=blue}
\numberwithin{equation}{section}

\newtheorem{theorem}{Theorem}[section]
\newtheorem{proposition}[theorem]{Proposition}
\newtheorem{lemma}[theorem]{Lemma}
\newtheorem{corollary}[theorem]{Corollary}
\newtheorem{question}[theorem]{Question}

\theoremstyle{definition}
\newtheorem{definition}[theorem]{Definition}
\newtheorem{example}[theorem]{Example}

\theoremstyle{remark}

\newcommand{\LP}{\mathcal{LP}}
\newcommand{\LPI}{\mathcal{LP}^{I}}
\newcommand{\qLPw}{\mathcal{LP}^{\mathrm{w}}_{q}}
\newcommand{\qLPIw}{\mathcal{LP}^{I,\mathrm{w}}_{q}}
\newcommand{\qLPs}{\mathcal{LP}^{\mathrm{s}}_{q}}
\newcommand{\qLPIs}{\mathcal{LP}^{I,\mathrm{s}}_{q}}
\newcommand{\Bq}{\mathcal{B}_{q}}
\newcommand{\R}{\mathbb{R}}
\newcommand{\C}{\mathbb{C}}
\newcommand{\N}{\mathbb{N}}

\begin{document}

\title[Weak and strong $q$-analogs of the Laguerre--P\'olya class]{Weak and strong $q$-analogs of the Laguerre--P\'olya class}

\author{Dimitar K. Dimitrov}
\address{Departamento de Matem\'atica Aplicada, IBILCE, Universidade Estadual Paulista,\\
15054-000 S\~ao Jos\'e do Rio Preto, SP, Brazil}
\email{dimitrov@ibilce.unesp.br}

\author{Boris Shapiro}
\address{Department of Mathematics, Stockholm University, SE-106 91 Stockholm, Sweden}
\email{shapiro@math.su.se}
\thanks{Corresponding author: Boris Shapiro, \texttt{shapiro@math.su.se}.}


\begin{abstract}
For $0<q<1$ we compare two natural $q$-analogs of the Laguerre--P\'olya class.  The first one is a coefficient-side class, defined as the inverse image of the classical Laguerre--P\'olya class under the normalized $q$-Borel transform
\[
  \Bq\left(\sum_{k\ge 0}a_k\frac{z^k}{k!}\right)
  =\sum_{k\ge 0}a_k\frac{q^{k(k-1)/2}(1-q)^k}{(q;q)_k}z^k .
\]
The second one is a zero-side class, defined as the locally uniform closure of real polynomials whose nonzero zeros are logarithmically $q$-separated on each side of the origin.  We prove that the normalized $q$-Borel transform maps the classical Laguerre--P\'olya class, and its type-I subclass, into themselves.  This yields a $q$-Jensen-polynomial criterion and shows that the coefficient-side class strictly contains the classical Laguerre--P\'olya class.  On the zero side, we prove a genus-zero product representation.  The logarithmic separation condition prevents zeros escaping to infinity from producing a residual exponential factor; consequently no nonconstant exponential factor can occur.  For every $q\in(0,1)$ we obtain the strict chains
\[
        \qLPs\subsetneq \LP\subsetneq \qLPw,
        \qquad
        \qLPIs\subsetneq \LPI\subsetneq \qLPIw .
\]
\end{abstract}

\maketitle

\section{Introduction}

A real entire function is said to belong to the Laguerre--P\'olya class, denoted by $\LP$, if it is the limit, uniformly on compact subsets of $\C$, of real polynomials with only real zeros.  Equivalently, it has the canonical product representation
\begin{equation}\label{eq:LPproduct}
 f(z)=c z^m e^{-a z^2+bz}
       \prod_{j=1}^{\omega}\left(1-\frac{z}{x_j}\right)e^{z/x_j},
\end{equation}
where $c,b,x_j\in\R$, $a\ge0$, $m\in\N\cup\{0\}$, $x_j\ne0$, and $\sum_j |x_j|^{-2}<\infty$; the product may be finite or empty.  We shall use this standard form together with the approximation theorem for $\LP$; see, for example, Boas \cite[Chapter VII]{Boas} or Levin \cite[Chapter VIII]{Levin}.

The Laguerre--P\'olya class of type I, denoted by $\LPI$, consists of the functions in $\LP$ which, after possibly replacing $z$ by $-z$, are locally uniform limits of real polynomials whose zeros are all nonpositive.  This is one of the classical zero-preserving subclasses considered already in the work of Laguerre and P\'olya--Schur; see Laguerre's collected works \cite{Laguerre}, P\'olya--Schur \cite{PolyaSchur}, and Obreschkoff's monograph \cite{Obreschkoff}.  Equivalently, up to the change $z\mapsto -z$, they have the form
\begin{equation}\label{eq:LPIproduct}
 f(z)=c z^m e^{bz}\prod_{j=1}^{\omega}\left(1+\frac{z}{x_j}\right),
 \qquad x_j>0,
 \qquad \sum_j x_j^{-1}<\infty,
\end{equation}
with $c\in\R$ and $b\ge0$.

The classical P\'olya--Schur theorem says that a real sequence $\{\gamma_k\}_{k\ge0}$ is a multiplier sequence, i.e. the diagonal operator
\[
       \sum c_k z^k\longmapsto \sum \gamma_k c_k z^k
\]
preserves real-rootedness of polynomials, if and only if the exponential generating function
\[
       \sum_{k=0}^{\infty}\gamma_k\frac{z^k}{k!}
\]
belongs to $\LPI$ after possibly replacing $z$ by $-z$; see \cite{PolyaSchur}.  We shall also use Jensen's characterization: if
\begin{equation}\label{eq:egf}
       f(z)=\sum_{k=0}^{\infty}a_k\frac{z^k}{k!}
\end{equation}
is real entire, then $f\in\LP$ if and only if all Jensen polynomials
\begin{equation}\label{eq:Jensen}
       J_n(f;z)=\sum_{k=0}^{n}\binom{n}{k}a_k z^k,\qquad n=0,1,2,\ldots,
\end{equation}
have only real zeros.

Let $0<q<1$.  We use the standard notation
\[
       (a;q)_n=\prod_{j=0}^{n-1}(1-aq^j),\qquad (a;q)_0=1,
\]
from the theory of basic hypergeometric series; see, for example, \cite{GR}.  We recall the big $q$-exponential
\begin{equation}\label{eq:Eq}
       E_q(z)=\sum_{k=0}^{\infty}\frac{q^{k(k-1)/2}}{(q;q)_k}z^k
             =(-z;q)_\infty
             =\prod_{j=0}^{\infty}(1+q^j z).
\end{equation}
Thus $E_q$ has zeros $-q^{-j}$, $j=0,1,2,\ldots , $ and belongs to $\LPI$.  This elementary fact is the main mechanism behind the coefficient-side part of the paper.

Given a real entire function $f$ written as in \eqref{eq:egf}, define its normalized $q$-Borel transform by
\begin{equation}\label{eq:Bq}
       (\Bq f)(z)
       :=\sum_{k=0}^{\infty} a_k
       \frac{q^{k(k-1)/2}(1-q)^k}{(q;q)_k}z^k.
\end{equation}
The normalization is chosen so that
\[
       \frac{q^{k(k-1)/2}(1-q)^k}{(q;q)_k}\longrightarrow \frac1{k!}
       \qquad(q\to1^-)
\]
for each fixed $k$; hence, under mild hypotheses which are automatic for entire functions, $\Bq f\to f$ locally uniformly as $q\to1^-$.  This leads to the following coefficient-side class.

\begin{definition}\label{def:weak}
For fixed $q\in(0,1)$, the weak $q$-Laguerre--P\'olya class is
\[
       \qLPw:=\{f\text{ real entire}:\ \Bq f\in\LP\}.
\]
Similarly,
\[
       \qLPIw:=\{f\text{ real entire}:\ \Bq f\in\LPI\}
\]
is the weak $q$-Laguerre--P\'olya class of type I.
\end{definition}

There is also a natural zero-side $q$-analog.  A finite or infinite set of nonzero real numbers is logarithmically $q$-separated if, on each side of the origin, consecutive moduli differ by at least the factor $q^{-1}$.  This condition is suggested by the zero set of $E_q$.  A closely related zero-theoretic framework appears in Lamprecht's work on Suffridge-type convolution theorems and $q$-extensions of P\'olya--Schur theory \cite{Lamprecht}; finite-difference analogs of P\'olya--Schur theory in an additive mesh setting were studied in \cite{BKS}.

For later comparison we also mention that several classical $q$-special functions have been studied from the point of view of real zeros and membership in the Laguerre--P\'olya class, for instance in \cite{Ismail,Nguyen}.  The zero-side class considered here is the following.

\begin{definition}\label{def:separated}
A finite or infinite sequence of nonzero real numbers is called $q$-separated if, for every two distinct entries $x$ and $y$ of the same sign with $0<|x|\le |y|$, one has
\[
        \frac{|x|}{|y|}\le q .
\]
Thus repeated nonzero entries are not allowed.  Zeros at the origin are allowed with arbitrary multiplicity and are not counted in the separation condition.

A real polynomial is called $q$-hyperbolic if all its nonzero zeros are real and, counted with multiplicity, form a $q$-separated sequence.  The strong $q$-Laguerre--P\'olya class $\qLPs$ is the locally uniform closure, in the space of real entire functions, of $q$-hyperbolic polynomials.  The class $\qLPIs$ is defined similarly, with all nonzero zeros required to have one sign.
\end{definition}

The main results of the paper can now be summarized as follows.  First, $\Bq$ preserves the classical Laguerre--P\'olya class and its type-I subclass.  Second, the weak class is strictly larger than the classical class.  Third, the strong class has an exact genus-zero product representation, and this rules out any nonconstant exponential factor.  These facts give the strict inclusions
\[
       \qLPs\subsetneq\LP\subsetneq\qLPw,
       \qquad
       \qLPIs\subsetneq\LPI\subsetneq\qLPIw
\]
for every $q\in(0,1)$.

Let us emphasize the distinction between the two constructions.  The weak class is governed by the coefficient multiplier sequence coming from the big $q$-exponential; its basic properties are therefore consequences of the classical P\'olya--Schur theorem together with Jensen's criterion.  The strong class is instead a closure of polynomials with a prescribed multiplicative mesh of zeros.  Its product representation is not a formal consequence of the classical Laguerre--P\'olya theorem: one must show that the mesh condition is closed under locally uniform limits and that it eliminates the exponential factors which are allowed in the classical canonical product.

\section{The coefficient-side class}

We start with two elementary estimates for the normalized coefficients
\begin{equation}\label{eq:deltak}
       \delta_k(q):=k!\frac{q^{k(k-1)/2}(1-q)^k}{(q;q)_k},\qquad k=0,1,2,\ldots .
\end{equation}
The exponential generating function of $\{\delta_k(q)\}$ is
\begin{equation}\label{eq:delta-egf}
       \sum_{k=0}^{\infty}\delta_k(q)\frac{z^k}{k!}
       =E_q((1-q)z)
       =\prod_{j=0}^{\infty}\left(1+(1-q)q^j z\right).
\end{equation}
Hence this generating function belongs to $\LPI$.

\begin{lemma}\label{lem:delta}
For every $0<q<1$, the sequence $\{\delta_k(q)\}_{k\ge0}$ is a multiplier sequence of type I.  Moreover
\[
       0<\delta_k(q)\le1,
       \qquad k=0,1,2,\ldots .
\]
\end{lemma}

\begin{proof}
The first assertion follows immediately from \eqref{eq:delta-egf} and the P\'olya--Schur theorem, since all zeros of $E_q((1-q)z)$ are real and negative.

For the estimate, write
\[
\delta_k(q)=\prod_{j=1}^{k}\frac{j q^{j-1}}{1+q+\cdots+q^{j-1}}.
\]
Since $1+q+\cdots+q^{j-1}\ge j q^{j-1}$ for $0<q<1$, each factor is at most one.
\end{proof}

\begin{lemma}\label{lem:Bq-continuity}
Let $0<q<1$.  If $f_n\to f$ locally uniformly in $\C$, where the $f_n$ are real entire functions, then $\Bq f_n\to\Bq f$ locally uniformly in $\C$.
\end{lemma}

\begin{proof}
Write
\[
       f_n(z)=\sum_{k=0}^{\infty}b_{n,k}z^k,
       \qquad
       f(z)=\sum_{k=0}^{\infty}b_k z^k.
\]
Then
\[
       \Bq f_n(z)=\sum_{k=0}^{\infty}\delta_k(q)b_{n,k}z^k,
       \qquad
       \Bq f(z)=\sum_{k=0}^{\infty}\delta_k(q)b_kz^k.
\]
Fix $0<R<S$ and put
\[
       M_n=\sup_{|z|\le S}|f_n(z)-f(z)|.
\]
By Cauchy's estimate, $|b_{n,k}-b_k|\le M_n S^{-k}$.  Hence, using Lemma \ref{lem:delta}, for $|z|\le R$,
\[
       |\Bq(f_n-f)(z)|
       \le \sum_{k=0}^{\infty}\delta_k(q)|b_{n,k}-b_k|R^k
       \le M_n\sum_{k=0}^{\infty}\left(\frac RS\right)^k.
\]
Since $M_n\to0$, the claim follows.
\end{proof}

\begin{theorem}\label{thm:Bq-preserves-LP}
For every $0<q<1$,
\[
       \Bq(\LP)\subseteq \LP,
       \qquad
       \Bq(\LPI)\subseteq \LPI.
\]
Equivalently,
\[
       \LP\subseteq\qLPw,
       \qquad
       \LPI\subseteq\qLPIw.
\]
\end{theorem}

\begin{proof}
Let
\[
       T_q\left(\sum_{k=0}^{n}c_k z^k\right)
       =\sum_{k=0}^{n}\delta_k(q)c_k z^k.
\]
By Lemma \ref{lem:delta} and the P\'olya--Schur theorem, $T_q$ preserves real-rootedness of polynomials.  More precisely, the generating function in \eqref{eq:delta-egf} has only negative zeros and belongs to $\LPI$; hence $\{\delta_k(q)\}$ is a type-I multiplier sequence.  Therefore $T_q$ maps polynomials with all zeros in $(-\infty,0]$ to polynomials with the same property.  Since every $\delta_k(q)$ is positive, applying the same statement to $p(-z)$ shows that $T_q$ also preserves the subclass of real-rooted polynomials whose zeros all lie in $[0,\infty)$.

If $f\in\LP$, choose real-rooted polynomials $p_n$ converging to $f$ locally uniformly.  Then $T_qp_n$ are real-rooted and, by Lemma \ref{lem:Bq-continuity}, $T_qp_n=\Bq p_n\to\Bq f$ locally uniformly.  Therefore $\Bq f\in\LP$.  The type-I assertion is identical, using approximation by polynomials whose zeros have one sign, after one fixed change $z\mapsto -z$ if necessary.
\end{proof}

\begin{theorem}[A $q$-Jensen criterion]\label{thm:qJensen}
Let $f(z)=\sum_{k=0}^{\infty}a_kz^k/k!$ be real entire and fix $q\in(0,1)$.  Then $f\in\qLPw$ if and only if, for every $n\ge0$, the polynomial
\begin{equation}\label{eq:qJensen}
       J_{n,q}(f;z):=
       \sum_{k=0}^{n}\binom{n}{k}a_k\,\delta_k(q)\,z^k
\end{equation}
has only real zeros.

Similarly, $f\in\qLPIw$ if and only if the polynomials $J_{n,q}(f;z)$ have only real zeros of one sign, after one common change $z\mapsto -z$ if necessary.
\end{theorem}

\begin{proof}
The ordinary Maclaurin coefficients of $\Bq f$ are $\delta_k(q)a_k/k!$.  Thus the Jensen polynomials of $\Bq f$ are precisely \eqref{eq:qJensen}.  Jensen's characterization of $\LP$ gives the first assertion.

For the type-I assertion, recall the standard type-I variant: a real entire function belongs to $\LPI$ if and only if, after possibly replacing $z$ by $-z$ once and for all, all its Jensen polynomials have only nonpositive zeros.  This follows by applying the classical Jensen approximation theorem to $g(z)$ or $g(-z)$ and using the P\'olya--Schur characterization of type-I multiplier sequences.  Applying this statement to $g=\Bq f$ gives the asserted criterion for $\qLPIw$.
\end{proof}

\begin{theorem}\label{thm:qto1}
For every real entire function $f$,
\[
       \Bq f\longrightarrow f
       \qquad\text{locally uniformly in }\C\text{ as }q\to1^-.
\]
\end{theorem}

\begin{proof}
Write $f(z)=\sum a_k z^k/k!$.  For fixed $k$,
\[
       \frac{q^{k(k-1)/2}(1-q)^k}{(q;q)_k}\to\frac1{k!}
       \qquad(q\to1^-).
\]
Moreover,
\[
       k!\frac{q^{k(k-1)/2}(1-q)^k}{(q;q)_k}=\delta_k(q)\le1
\]
by Lemma \ref{lem:delta}.  Hence, on $|z|\le R$,
\[
       \left|a_k\frac{q^{k(k-1)/2}(1-q)^k}{(q;q)_k}z^k\right|
       \le |a_k|\frac{R^k}{k!}.
\]
The majorant is summable because $f$ is entire.  Dominated convergence gives the result.
\end{proof}

\begin{theorem}\label{thm:weak-strict}
For every $q\in(0,1)$,
\[
       \LP\subsetneq\qLPw,
       \qquad
       \LPI\subsetneq\qLPIw.
\]
\end{theorem}

\begin{proof}
Consider the quadratic polynomial
\begin{equation}\label{eq:strict-example}
       F_q(z)=1+2z+\frac{1+q}{2q}z^2.
\end{equation}
In the notation $F_q(z)=\sum a_kz^k/k!$, one has
\[
       a_0=1,
       \qquad a_1=2,
       \qquad a_2=\frac{1+q}{q}.
\]
Therefore
\[
       \Bq F_q(z)=1+2z+z^2=(1+z)^2.
\]
Thus $F_q\in\qLPIw\subseteq\qLPw$.  On the other hand, the discriminant of $F_q$ is
\[
       4-4\frac{1+q}{2q}
       =-\frac{2(1-q)}{q}<0.
\]  Hence $F_q$ has two nonreal zeros and does not belong to $\LP$.  A fortiori it does not belong to $\LPI$.
\end{proof}

\section{The zero-side class}

We now prove the product representation for $\qLPs$.  The key point is that logarithmic $q$-separation prevents zeros escaping to infinity from producing a residual exponential factor.  This is the only delicate point in the zero-side theory.

\begin{lemma}\label{lem:tail}
Let $p$ be a $q$-hyperbolic polynomial.  For $R>0$,
\[
       \sum_{\substack{x:p(x)=0\\ |x|\ge R}}\frac1{|x|}
       \le \frac{2}{(1-q)R},
\]
where the sum is over nonzero zeros and counts multiplicities.  More generally, if $s>0$, then
\[
       \sum_{\substack{x:p(x)=0\\ |x|\ge R}}\frac1{|x|^s}
       \le \frac{2}{(1-q^s)R^s}.
\]
\end{lemma}

\begin{proof}
On the positive half-axis, list the zeros not smaller than $R$ as
\[
       R\le x_1<x_2<\cdots .
\]
The $q$-separation condition gives $x_{j+1}\ge q^{-1}x_j$, hence $x_j\ge Rq^{-(j-1)}$.  Therefore
\[
       \sum_j x_j^{-s}\le R^{-s}\sum_{j=0}^{\infty}q^{js}=\frac1{(1-q^s)R^s}.
\]
The negative half-axis gives the same bound.
\end{proof}

\begin{lemma}[Closedness of the multiplicative mesh]\label{lem:closed-mesh}
Let $0<q<1$, and let $p_n$ be $q$-hyperbolic polynomials converging locally uniformly in $\C$ to a real entire function $f\not\equiv0$.  Then all zeros of $f$ are real.  Moreover, every nonzero zero of $f$ is simple, and the nonzero zeros of $f$ form a $q$-separated set.
\end{lemma}

\begin{proof}
The assertion that all zeros of $f$ are real follows from Hurwitz's theorem: a small disk centered at a nonreal zero and disjoint from the real axis would have to contain zeros of $p_n$ for all sufficiently large $n$, which is impossible.

Let $x\ne0$ be a zero of $f$ of multiplicity $r$.  Choose $\varepsilon>0$ so small that the interval $(x-\varepsilon,x+\varepsilon)$ does not meet the origin and that all its points have the same sign.  Hurwitz's theorem, applied on a small circle around $x$, implies that for all sufficiently large $n$ the polynomial $p_n$ has exactly $r$ zeros, counted with multiplicity, in this interval.  If $r\ge2$, two of these zeros have moduli whose quotient tends to $1$, contradicting the $q$-separation condition for all large $n$.  Hence $r=1$.

It remains to check separation between distinct limiting zeros.  Let $x$ and $y$ be two distinct zeros of $f$ with the same sign and with $0<|x|\le |y|$.  If $|x|/|y|>q$, choose disjoint intervals $I_x$ and $I_y$, centered at $x$ and $y$ respectively and contained in the same half-axis, such that
\[
        \frac{|u|}{|v|}>q
        \qquad (u\in I_x,\, v\in I_y,\ 0<|u|\le |v|).
\]
Hurwitz's theorem gives zeros $x_n\in I_x$ and $y_n\in I_y$ of $p_n$ for all large $n$, contradicting the $q$-separation of the zeros of $p_n$.  Therefore $|x|/|y|\le q$, as required.
\end{proof}

\begin{theorem}[Product representation for the strong class]\label{thm:strong-product}
Let $0<q<1$.  A nonzero real entire function $f$ belongs to $\qLPs$ if and only if it has the representation
\begin{equation}\label{eq:strong-product}
       f(z)=c z^m\prod_{j=1}^{\omega}\left(1-\frac{z}{x_j}\right),
\end{equation}
where $c\in\R\setminus\{0\}$, $m\in\N\cup\{0\}$, the numbers $x_j\in\R\setminus\{0\}$ are $q$-separated, and
\[
       \sum_j\frac1{|x_j|}<\infty.
\]
The zero function belongs to $\qLPs$ by definition.
\end{theorem}

\begin{proof}
First suppose that $f\in\qLPs$ and $f\not\equiv0$.  Choose $q$-hyperbolic polynomials $p_n$ such that $p_n\to f$ locally uniformly.  By Lemma \ref{lem:closed-mesh}, all zeros of $f$ are real, every nonzero zero is simple, and the nonzero zeros of $f$ are $q$-separated.  On the positive half-axis their moduli, if infinite in number, therefore grow at least geometrically; the same is true on the negative half-axis.  Hence
\[
       \sum_j |x_j|^{-1}<\infty .
\]

It remains to prove that no nonconstant exponential factor is left after passing to the limit.  We first treat the case $f(0)\ne0$.  Then $p_n(0)\ne0$ for all sufficiently large $n$, and
\[
       p_n(z)=p_n(0)\prod_x\left(1-\frac{z}{x}\right),
\]
where the product is over the nonzero zeros of $p_n$.  Fix $\rho>0$.  Choose $R>2\rho$ such that $f$ has no zero on $|z|=R$.  Let $P_{n,R}$ be the finite product formed from the zeros of $p_n$ in $|z|<R$, and let $T_{n,R}$ be the remaining tail product.  Thus
\[
       p_n(z)=p_n(0)P_{n,R}(z)T_{n,R}(z).
\]
By Hurwitz's theorem, for this fixed $R$ the zeros of $p_n$ in $|z|<R$ converge, as a multiset, to the zeros of $f$ in $|z|<R$.  Consequently
\[
       p_n(0)P_{n,R}(z)\longrightarrow f(0)P_R(z)
\]
locally uniformly, where
\[
       P_R(z)=\prod_{\substack{j:\ |x_j|<R}}\left(1-\frac{z}{x_j}\right).
\]
In particular $p_n(0)P_{n,R}$ is uniformly bounded on $|z|\le\rho$ for all large $n$.

For $|z|\le\rho$ and any tail zero satisfying $|x|\ge R$, we have $|z/x|<1/2$ and hence, with the principal branch of the logarithm,
\[
       \left|\log\left(1-\frac{z}{x}\right)\right|\le \frac{2\rho}{|x|}.
\]
Lemma \ref{lem:tail} gives the uniform estimate
\[
       \left|\log T_{n,R}(z)\right|\le \frac{4\rho}{(1-q)R},
       \qquad |z|\le\rho .
\]
Therefore
\[
       \sup_{|z|\le\rho}|T_{n,R}(z)-1|\le
       \exp\left(\frac{4\rho}{(1-q)R}\right)-1,
\]
uniformly in $n$.  Passing first to the limit $n\to\infty$ and then letting $R\to\infty$, we obtain
\[
       f(z)=f(0)\prod_j\left(1-\frac{z}{x_j}\right)
\]
locally uniformly on $|z|\le\rho$.  Since $\rho$ was arbitrary, this holds locally uniformly in $\C$.

Now suppose that $f$ has a zero of order $m$ at the origin.  Choose $\varepsilon>0$ so small that $f$ has no zeros in $0<|z|\le2\varepsilon$.  By Hurwitz's theorem, for all sufficiently large $n$ the polynomial $p_n$ has exactly $m$ zeros, counted with multiplicity, in $|z|<\varepsilon$ and no zeros on $|z|=\varepsilon$.  Let
\[
        A_n(z)=\prod_{\nu=1}^{m}(z-\alpha_{n,\nu})
\]
be the monic factor formed from these zeros, and write $p_n=A_n g_n$.  Then $\alpha_{n,\nu}\to0$, so $A_n\to z^m$ locally uniformly.  On the circle $|z|=\varepsilon$ the functions $g_n=p_n/A_n$ converge uniformly to $f(z)/z^m$.  Since the $g_n$ are holomorphic in $|z|\le\varepsilon$, the maximum principle gives the same convergence in the disk.  On compact subsets of $\C\setminus\{0\}$ the convergence follows directly from $p_n\to f$ and $A_n\to z^m$.  Hence $g_n\to f(z)/z^m$ locally uniformly in $\C$.

The nonzero zeros of each $g_n$ form a subsequence of the nonzero zeros of $p_n$; hence $g_n$ is again $q$-hyperbolic.  Since $f(z)/z^m$ does not vanish at the origin, the already proved case applies to this limit and gives the required product representation for $f$.

Conversely, if $f$ has the representation \eqref{eq:strong-product}, then the partial products
\[
       c z^m\prod_{j=1}^{N}\left(1-\frac{z}{x_j}\right)
\]
are $q$-hyperbolic polynomials and converge locally uniformly to $f$, because $\sum_j |x_j|^{-1}<\infty$.  Hence $f\in\qLPs$.
\end{proof}

\begin{corollary}\label{cor:strong-typeI}
A nonzero real entire function belongs to $\qLPIs$ if and only if, after possibly replacing $z$ by $-z$, it has the representation
\[
       f(z)=c z^m\prod_{j=1}^{\omega}\left(1+\frac{z}{x_j}\right),
       \qquad x_j>0,
\]
where the sequence $\{x_j\}$ is $q$-separated and $\sum_j x_j^{-1}<\infty$.
\end{corollary}

\begin{corollary}\label{cor:strong-in-LP}
For every $q\in(0,1)$,
\[
       \qLPs\subseteq\LP,
       \qquad
       \qLPIs\subseteq\LPI.
\]
More precisely, every nonzero function in $\qLPs$ has genus zero and no nonconstant exponential factor.
\end{corollary}

\begin{proof}
This is immediate from Theorem \ref{thm:strong-product} and the classical product representation of $\LP$.
\end{proof}

\begin{proposition}\label{prop:not-multiplicative}
For every $q\in(0,1)$, the class $\qLPs$ is not closed under multiplication.
\end{proposition}

\begin{proof}
Choose $y$ with $1<y<q^{-1}$.  The polynomials $1-z$ and $1-z/y$ each belong to $\qLPs$, since each has only one nonzero zero.  Their product has two positive zeros, $1$ and $y$, which are not $q$-separated.  By Theorem \ref{thm:strong-product}, the product does not belong to $\qLPs$.
\end{proof}

\section{Comparison of the two classes}

The results above give the precise inclusion chain.

\begin{theorem}\label{thm:chain}
For every $q\in(0,1)$,
\[
       \qLPs\subsetneq \LP\subsetneq \qLPw.
\]
Likewise,
\[
       \qLPIs\subsetneq \LPI\subsetneq \qLPIw.
\]
\end{theorem}

\begin{proof}
The inclusions $\qLPs\subseteq\LP$ and $\qLPIs\subseteq\LPI$ follow from Corollary \ref{cor:strong-in-LP}.  They are proper because $e^{-z^2}\in\LP$ but, by Theorem \ref{thm:strong-product}, a zero-free function in $\qLPs$ must be constant; and $e^{z}\in\LPI$ but a zero-free function in $\qLPIs$ must again be constant.

The inclusions $\LP\subseteq\qLPw$ and $\LPI\subseteq\qLPIw$ are Theorem \ref{thm:Bq-preserves-LP}.  They are proper by Theorem \ref{thm:weak-strict}.
\end{proof}

\begin{example}\label{ex:Eq}
The function
\[
       E_q((1-q)z)=\prod_{j=0}^{\infty}\left(1+(1-q)q^jz\right)
\]
belongs to $\qLPIs$, because its zeros are
\[
       -\frac{q^{-j}}{1-q},\qquad j=0,1,2,\ldots,
\]
which form a $q$-geometric progression on the negative axis.  It also belongs to $\qLPIw$ by Theorem \ref{thm:Bq-preserves-LP}, since it belongs to $\LPI$ classically.
\end{example}

\begin{example}\label{ex:weak-not-strong}
For every $q\in(0,1)$, the function $e^{-z^2}$ belongs to $\qLPw$ by Theorem \ref{thm:Bq-preserves-LP}, but it does not belong to $\qLPs$ by Theorem \ref{thm:strong-product}.  Similarly, $e^z\in\qLPIw$ but $e^z\notin\qLPIs$.
\end{example}

\begin{example}\label{ex:sin}
The function $\sin z$ belongs to $\LP$, hence to $\qLPw$, for every $q\in(0,1)$.  It is not in $\qLPs$, because its positive zeros $\pi,2\pi,3\pi,\ldots$ are not logarithmically $q$-separated.
\end{example}

\section{Weak \texorpdfstring{$q$}{q}-multiplier sequences}

The coefficient-side definition suggests a cautious terminology for coefficient sequences.

\begin{definition}\label{def:qmult}
Let $0<q<1$.  A real sequence $\{a_k\}_{k\ge0}$ is called a weak $q$-multiplier sequence if the associated $q$-exponential generating function
\[
       \sum_{k=0}^{\infty}a_k\frac{q^{k(k-1)/2}(1-q)^k}{(q;q)_k}z^k
\]
belongs to $\LP$, whenever this series defines an entire function.
\end{definition}

This definition is meant only as coefficient-side terminology; it is not, by itself, an intrinsic preservation theorem for an operator on a zero-side class.  With this terminology, Theorem \ref{thm:Bq-preserves-LP} says that every classical multiplier sequence gives a weak $q$-multiplier sequence after the normalization \eqref{eq:Bq}.  A full P\'olya--Schur-type classification of weak or strong $q$-multiplier sequences remains open.  Lamprecht's convolution theory \cite{Lamprecht} and the finite-difference mesh-preserving theory of \cite{BKS} suggest that the zero-side version should be genuinely more rigid than the coefficient-side version.

\section{Questions}

We close with some questions which, in the present formulation, are deliberately separated from the proved results.

\begin{question}
Find an intrinsic coefficient characterization of $\qLPw$ which does not simply restate the condition $\Bq f\in\LP$.  In particular, is there a useful analog of the classical Laguerre inequalities for the coefficients of functions in $\qLPw$?
\end{question}

\begin{question}
Is $\qLPw$ closed under differentiation?  Equivalently, if $\Bq f\in\LP$, must $\Bq(f')\in\LP$?
\end{question}

\begin{question}
Characterize diagonal operators preserving $\qLPs$ or preserving the set of $q$-hyperbolic polynomials.  Is there a satisfactory multiplicative-mesh analog of the P\'olya--Schur theorem?
\end{question}

\begin{question}
Which classical $q$-special functions belong to $\qLPw$, to $\qLPs$, or to both?  In particular, can the known real-rootedness results for the Ramanujan function and related $q$-Bessel functions, such as those discussed in \cite{Ismail,Nguyen}, be reinterpreted naturally in this framework?
\end{question}

\begin{question}
Study the limiting behavior of the strong class as $q\to1^-$ and $q\to0^+$.  The weak class converges to the classical class in the operator sense of Theorem \ref{thm:qto1}; the corresponding zero-side limit is subtler because the $q$-separation condition degenerates as $q\to1^-$.
\end{question}

%\section*{Statements and Declarations}

%\noindent\textbf{Funding.}
%The research of the first author was supported by the Brazilian Science Foundations FAPESP under Grant 97/6280-0 and CNPq under Grant 300645/95-3.

%\medskip
%\noindent\textbf{Competing interests.}
%The authors have no relevant financial or non-financial interests to disclose.

%\medskip
%\noindent\textbf{Data availability.}
%Data sharing is not applicable to this article, as no datasets were generated or analyzed during the current study.

%\medskip
%\noindent\textbf{Author contributions.}
%Both authors contributed to the mathematical conception, development, and writing of the manuscript.  Both authors read and approved the final manuscript.

\end{document}